\documentclass[12pt,a4paper]{article}

\usepackage{authblk}
\usepackage[margin=1.8cm,bottom=2cm]{geometry}
\usepackage{t1enc}
\usepackage[utf8]{inputenc}
\usepackage{amsmath,amsthm,amssymb}
\usepackage{graphicx}
\usepackage{enumerate}
\usepackage{hyperref}
\usepackage{bm}
\usepackage{comment}
\usepackage{amsfonts}
\usepackage{graphicx,caption}
\usepackage{bm}
\usepackage{amsmath, amsthm, amssymb}
\usepackage{graphicx}
\usepackage{hyperref}
\usepackage{relsize}
\usepackage{blkarray}
\usepackage{enumitem}
\usepackage{algpseudocode}

\usepackage{bbm}

\theoremstyle{plain}
\usepackage{amsthm}
\makeatletter
\newcommand{\newreptheorem}[2]{\newtheorem*{rep@#1}{\rep@title}\newenvironment{rep#1}[1]{\def\rep@title{#2 \ref*{##1}}\begin{rep@#1}}{\end{rep@#1}}}
\makeatother

\newtheorem{theorem}{Theorem}[section]
\newtheorem*{theorem-non}{Theorem}
\newtheorem*{non-lemma}{Lemma}
\newtheorem{lemma}[theorem]{Lemma}
\newreptheorem{lemma}{Lemma}

\newtheorem{claim}[theorem]{Claim}

\theoremstyle{definition}

\DeclareMathOperator{\Sig}{Sig}
\DeclareMathOperator{\rang}{rank}

\DeclareMathOperator{\tyg}{Type}

\DeclareMathOperator{\RowS}{RowSpace}

\DeclareMathOperator{\cok}{cok}

\DeclareMathOperator{\Hom}{Hom}
\DeclareMathOperator{\Aut}{Aut}

\DeclareMathOperator{\Sg}{Sub}

\begin{document}
\title{Universal constant order fluctuations for the cokernels of block triangular matrices}
\author{Andr\'as M\'esz\'aros}
\date{}
\affil{HUN-REN Alfr\'ed R\'enyi Institute of Mathematics, Budapest, Hungary}
\maketitle
\begin{abstract}
We prove that for a large class of random block lower triangular matrices, the Sylow \break $p$-subgroups of their cokernels have the same constant order fluctuations as that of the matrix products studied by Nguyen and Van Peski~\cite{nguyen2024rank}. We also show that the theorem of Nguyen and Van Peski remains true under a weaker assumption on the number of factors in the matrix products.
\end{abstract}

\section{Introduction}

The \emph{cokernel} of an $n\times n$ integral matrix $\mathbf{M}$ is defined as
\[\cok(\mathbf{M})=\mathbb{Z}^n/\RowS(\mathbf{M}),\]
where $\RowS(\mathbf{M})$ is the subgroup of $\mathbb{Z}^n$ generated by the rows of $\mathbf{M}$. When we study the cokernels of random matrices, we are often interested in the Sylow $p$-subgroup of the cokernel. There is a widespread \emph{universality} phenomenon for the cokernels of random matrices, that is, in many cases the limiting behavior of the Sylow $p$-subgroup of the cokernel does not depend on the finer details of the given random matrix model. The simplest form of this universality was proved by Wood~\cite{wood2019random}: Given $\varepsilon>0$ and a prime $p$, we say that a $\mathbb{Z}$-valued random variable $X$ is $(p,\varepsilon)$-balanced if $\mathbb{P}(X\equiv r\mod p)\le 1-\varepsilon$ for all $r\in \mathbb{Z}$. Wood proved that if $\mathbf{M}_n$ is a random $n\times n$ matrix with independent $(p,\varepsilon)$-balanced entries, then the distribution of the Sylow $p$-group of $\cok(\mathbf{M}_n)$ converges to the \emph{Cohen-Lenstra distribution}. The Cohen-Lenstra distribution is a distribution on the finite abelian $p$-groups, where the probability assigned to a group $G$ is proportional to $|\Aut(G)|^{-1}$. This distribution first appeared in conjectures about the distribution of class groups of quadratic number fields~\cite{cohen2006heuristics}. Note that this limiting distribution does not depend on the specific distribution of the entries. Nguyen and Wood~\cite{nguyen2022random} weakened the assumptions of the theorem above: It is enough to assume that the entries of $\mathbf{M}_n$ are $(p,\varepsilon_n)$-balanced where $\varepsilon_n$ converges to $0$ sufficiently slowly. The proofs of these theorems rely on the \emph{moment method} of Wood for random finite abelian groups~\cite{wood2017distribution,wood2022probability}. 

If a random matrix model has some additional algebraic structure, then this algebraic structure is also reflected in the distribution of the cokernel. For example, for symmetric matrices, the cokernel can be equipped with a perfect, symmetric, bilinear pairing. Thus, to obtain the limiting distribution for the cokernels of symmetric matrices, one needs to modify the Cohen-Lenstra distribution by taking into account the number of such pairings~\cite{clancy2015note,clancy2015cohen,wood2017distribution}. A similar phenomenon occurs for skew symmetric matrices~\cite{delaunay2001heuristics,bhargava2015modeling,nguyen2022local}, and for $p$-adic Hermitian matrices~\cite{lee2023universality}. We again see that the limiting behavior of the cokernel is universal: it depends on the symmetry class of the random matrix model, but not on the specific distribution of the entries. See \cite{recent2,recent3,recent4,recent5,recent6,recent7,recent8,meszaros2020distribution,meszaros2023cohen,meszaros2024Zpband,gorokhovsky2024time,kang2024random} for further results on the cokernels of random matrices.

Nguyen and Van Peski~\cite{nguyen2024universality,nguyen2024rank} considered \emph{matrix products} of the form $\mathbf{A}_1\mathbf{A}_2\cdots \mathbf{A}_k$, where $\mathbf{A}_1,\dots,\mathbf{A}_k$ are i.i.d. $n\times n$ random matrices, where the entries are i.i.d. copies of some $\mathbb{Z}$-valued random variable which is non-constant mod $p$. Let $\Gamma$ be the Sylow $p$-group of $\cok(\mathbf{A}_1\mathbf{A}_2\cdots \mathbf{A}_k)$. For fixed~$k$, Nguyen and Van Peski showed that as $n\to\infty$, $\Gamma$ has a limiting distribution~\cite{nguyen2024universality}. This limiting distribution is a modified version of the Cohen-Lenstra distribution where the probability assigned to a group $G$ is obtained by also taking into account the number of chains of subgroups of~$G$ of length~$k$. When $k$ goes to infinity together with $n$, the size of $\Gamma$ goes to infinity. Thus, $\Gamma$ can not have a limiting distribution in the usual sense. However, Nguyen and Van Peski~\cite{nguyen2024rank} proved that if $n,k\to \infty$ such that $k=O(e^{(\log n)^{1-\varepsilon}})$, then the fluctuation of $\Gamma$ stays of constant order. They also determined the limiting distribution of these fluctuations. The limit again does not depend on the distribution of the entries. See Theorem~\ref{nguyenvanPeski} for a more detailed statement. This is a new interesting phenomenon, so it is a natural question whether we can see the same type of constant order fluctuations for other random groups of growing order.

Let us consider the block lower bidiagonal matrix
\begin{equation}\label{Cconstruct}\mathbf{C}=\begin{pmatrix}
\mathbf{A}_1&\mathbf{0}&\ddots&&&\\
\mathbf{I}&\mathbf{A}_2&\mathbf{0}&\ddots&\\
\mathbf{0}&\mathbf{I}&\mathbf{A}_3&\mathbf{0}&\ddots&\\
\ddots&\mathbf{0}&\mathbf{I}&\mathbf{A}_4&\mathbf{0}&\ddots&\\
&\ddots&\ddots&\ddots&\ddots&\ddots\\
& &\ddots&\mathbf{0}&\mathbf{I}&\mathbf{A}_k
\end{pmatrix}.
\end{equation}

One can prove that
\begin{equation}\label{cokidentity}
\cok(\mathbf{C})\cong \cok(\mathbf{A}_1\mathbf{A}_2\cdots \mathbf{A}_k),
\end{equation}
see Section~\ref{appendix}. Thus, $\cok(\mathbf{C})$ shows the same behavior as the cokernel of matrix products. Given the prevalence of the universality phenomenon in the study of cokernels of random matrices, one could expect the same behavior from many other random \emph{block lower triangular} matrix models. Theorem~\ref{thmmain} below provides such examples of block lower triangular matrices. Therefore, the hidden block lower triangular structure of matrix products gives a more general context where the constant order fluctuations found by Nguyen and Van Peski can be observed universally. 

Now we describe our setting:

Let $\mathcal{P}$ be a subset of primes, and let $\varepsilon>0$. We say that a $\mathbb{Z}$-valued random variable $X$ is $(\mathcal{P},\varepsilon)$-balanced if 
\[\mathbb{P}(X\equiv r\mod p)\le 1-\varepsilon\qquad\text{ for all }p\in \mathcal{P}\text{ and } r\in\mathbb{Z}.\]
We say that a random matrix over $\mathbb{Z}$ is $(\mathcal{P},\varepsilon)$-balanced if it has independent $(\mathcal{P},\varepsilon)$-balanced entries.

Let $\mathbf{A}^{(h)}=(\mathbf{A}^{(h)}_{i,j})_{1\le i,j\le k(h)}$ and $\mathbf{B}^{(h)}=(\mathbf{B}^{(h)}_{i,j})_{1\le i,j\le k(h)}$ be sequences of random block matrices over the integers such that the blocks $\mathbf{A}^{(h)}_{i,j}$ and $\mathbf{B}^{(h)}_{i,j}$ are $n_i(h)\times n_j(h)$ matrices. 

Let $\mathbf{C}^{(h)}=\mathbf{A}^{(h)}+\mathbf{B}^{(h)}$. We also define \[n_\bullet(h)=\min_{i=1}^{k(h)} n_i(h)\qquad\text{ and }\qquad n(h)=\sum_{i=1}^{k(h)} n_i(h).\]

Furthermore, assume that:
\begin{enumerate}[label=(A.\arabic*)]
\item \label{A1} $\mathbf{A}^{(h)}$ and $\mathbf{B}^{(h)}$ are independent;
\item \label{A2} If $j\notin\{i-1,i\}$, then $\mathbf{A}^{(h)}_{i,j}=0$;
\item \label{A3} The blocks $\mathbf{A}^{(h)}_{i,j}$ ($i\in\{1,2,\dots,k(h)\}$, $j\in\{i-1,i\}$) are independent $(\mathcal{P},\varepsilon)$-balanced random matrices;
\item \label{A4} If $j\ge i-1$, then $\mathbf{B}^{(h)}_{i,j}=0$;
\item \label{A5} \[\lim_{h\to\infty} k(h)=\infty;\]
\item \label{A6} \[\lim_{h\to\infty} n_\bullet(h)=\infty;\]
\item \label{A7} \[\lim_{h\to\infty} \frac{\log k(h)}{n_\bullet(h)}=0.\]
\end{enumerate}

Given a finite abelian group $G$ and a nonnegative integer $i$, let $c(G,i)$ be the number of chains of subgroups $\{0\}\subsetneq H_1\subsetneq H_2\subsetneq \cdots\subsetneq H_i\subseteq G$ of length $i$. The largest possible length $\ell(G)$ of such a chain can be obtained as $\ell(G)=\sum e_i$, where $|G|=\prod p_i^{e_i}$ is the prime factorization of~$|G|$.

\begin{theorem}\label{homlemma}
 Let $G$ be a finite abelian group, let $\mathcal{P}$ be the set of prime divisors of $|G|$. Then for any sequence of random matrices $\mathbf{C}^{(h)}$ satisfying the assumptions \ref{A1}--\ref{A7} above, we have
 \[\lim_{h\to\infty} \frac{\mathbb{E}|\Hom(\cok(\mathbf{C}^{(h)}),G)|}{k(h)^{\ell(G)}}=\frac{c(G,\ell(G))}{\ell(G)!}.\]
\end{theorem}

The expectations $\mathbb{E}|\Hom(\cok(\mathbf{C}^{(h)}),G)|$ are usually called the Hom-moments \break of $\cok(\mathbf{C}^{(h)})$. Note that the limiting rescaled moments do not depend on the finer details of our random matrix models such as the sizes of the blocks or the exact distribution of the entries. This is another manifestation of the universality phenomenon.

Combining Theorem~\ref{homlemma} with a special case of the \emph{rescaled moment method} of Nguyen and Van Peski~\cite{nguyen2024rank} stated in Theorem~\ref{rescaledmomentmethod}, one can describe the limiting fluctuations of the Sylow $p$-subgroup of $\cok(\mathbf{C}^{(h)})$, as we explain next.

Let $p$ be a prime. Any partition $\lambda=(\lambda_1,\lambda_2,\dots,\lambda_r)$ gives rise to a finite abelian $p$-group
\[G_{\lambda}=\bigoplus_{i=1}^r \mathbb{Z}/p^{\lambda_i}\mathbb{Z}.\]

This gives a one-to-one correspondence between partitions and finite abelian $p$-groups. Note that if $\lambda'$ is the conjugate of the partition $\lambda$, that is, $\lambda'_i=|\{j:\lambda_j\ge i\}|$, then $\lambda_i'$ also has a group theoretic meaning, namely,
\begin{equation}\label{conjugate}\lambda_i'=\text{rank}(p^{i-1} G_\lambda).\end{equation}
Let $|\lambda|=\sum_{i=1}^d \lambda_i$. Note that $\ell(G_{\lambda})=\ell(G_{\lambda'})=|\lambda|$, and $|G_{\lambda}|=|G_{\lambda'}|=p^{|\lambda|}$.

Thus, the Sylow $p$-subgroup of the cokernel of an integral random matrix, which is a finite abelian $p$-group can be also represented by a partition, or equivalently by the corresponding Young diagram. Thus, a random integral matrix determines a random Young diagram. Theorem~\ref{thmmain} below describes the joint fluctuations of the first $d$ columns of the Young diagram corresponding the Sylow $p$-group of $\cok(\mathbf{C}^{(h)})$. Before we can state this theorem, we need a few further notions.

Given a positive integer $d$, let
\begin{align*}\Sig_d&=\{(\lambda_1,\lambda_2,\dots,\lambda_d)\in \mathbb{Z}^d\,:\, \lambda_1\ge\lambda_2\cdots\ge \lambda_d\}\text{ and }\\
\Sig_d^{\ge 0}&=\{(\lambda_1,\lambda_2,\dots,\lambda_d)\in \mathbb{Z}^d\,:\, \lambda_1\ge\lambda_2\cdots\ge \lambda_d\ge 0\}.
\end{align*}

Note that $\Sig_d^{\ge 0}$ is just the set of partitions with at most $d$ parts.

For two vectors $v,w\in \mathbb{Z}^d$, their inner product is defined as $\langle v,w\rangle=\sum_{i=1}^d v_iw_i$.

For $\chi>0$, let $\mathcal{L}_{d,p^{-1},\chi}$ be a $\Sig_d\,$-valued random variable such that
\[\mathlarger{\mathbb{E} p^{\left\langle \mathcal{L}_{d,p^{-1},\chi}\,,\,\lambda \right\rangle}}=\frac{((p-1)\chi)^{|\lambda|}}{|\lambda|!}c(G_{\lambda'},|\lambda|)\qquad \text{ for all }\lambda\in\Sig_d^{\ge 0}.\]
It was proved by Nguyen and Van Peski that such a random variable exists and its distribution is uniquely determined~\cite{nguyen2024rank}. These random variables first appeared in the work of Van Peski on the reflecting Poisson sea~\cite{van2023local,van2023reflecting}.

For $x\in \mathbb{R}$, let $\lfloor x \rceil$ be the integer closest to $x$. (If this is not unique, it does not matter which one we choose.)
\begin{theorem}\label{thmmain}
 Let $p$ be a prime, and let $\mathcal{P}=\{p\}$. Consider a sequence of random matrices $\mathbf{C}^{(h)}$ satisfying the assumptions \ref{A1}--\ref{A7} above. Furthermore, assume that the fractional part $\{-\log_p(k(h))\}$ converges to $\zeta$. Let $\chi=p^{-\zeta}/(p-1)$. Let $\Gamma_h$ be the Sylow $p$-subgroup of $\cok(\mathbf{C}^{(h)})$. 
 
 Then, for any $d\ge 1$, 
 \[\left(\rang(p^{i-1} \Gamma_h)-\lfloor \log_p(k(h))+\zeta \rceil \right)_{i=1}^{d}\]
 converges to $\mathcal{L}_{d,p^{-1},\chi}$ in distribution as $h\to\infty$.
\end{theorem}

By \eqref{conjugate} the theorem above describes the fluctuations of the first $d$ columns of the Young diagrams corresponding the Sylow $p$-group of $\cok(\mathbf{C}^{(h)})$. 

Comparing Theorem~\ref{thmmain} with the theorem of Nguyen and Van Peski~\cite{nguyen2024rank} below, we see that the cokernel of block lower triangular matrices lies in the same universality class as the cokernel of matrix products.

\begin{theorem}[Nguyen and Van Peski~\cite{nguyen2024rank}] \label{nguyenvanPeski}
Fix a prime $p$, and let $\xi$ be a $\mathbb{Z}$-valued random variable such that $\xi$ (mod p) is nonconstant. Let $k_n$ be a sequence of positive integers tending to infinity such that $k_n=O(e^{(\log n)^{1-\varepsilon}})$ for some $\varepsilon>0$. Let $\mathbf{M}_n$ be the product of $k_n$ iid $n\times n$ matrices with iid $\xi$ entries. Let $\Gamma_n$ be the Sylow $p$-group of $\cok(\mathbf{M}_n)$. 
Let $(n_j)$ be a subsequence such that the fractional part $\{-\log_p (k_{n_j})\}$ converges to $\zeta$. Let $\chi=p^{-\zeta}/(p-1)$. 

Then for all $d\ge 1$,
\[\left(\rang(p^{i-1}\Gamma_{n_j})-\lfloor\log_p(k_{n_j})+\zeta\rceil\right)_{i=1}^d\]
converges to $\mathcal{L}_{d,p^{-1},\chi}$ in distribution as $j\to\infty$.
\end{theorem}

Although Theorem~\ref{thmmain} does not cover the matrices obtained from matrix products by the construction given in \eqref{Cconstruct}, with only small modifications, our methods can also be applied to such matrices. In fact, our proof provides an improvement on Theorem~\ref{nguyenvanPeski} by replacing the condition on $k_n$ with a weaker one.
\begin{theorem}\label{improvedtheorem}
Theorem~\ref{nguyenvanPeski} remains valid, if we replace the condition $k_n=O(e^{(\log n)^{1-\varepsilon}})$ with
\[\lim_{n\to\infty} \frac{\log(k_n)}{n}=0.\]
\end{theorem}

\bigskip

\textbf{Acknowledgments:} The author is grateful to the anonymous referees for their useful suggestions. The author was supported by the KKP 139502 project and  Dynasnet European Research Council Synergy project -- grant number ERC-2018-SYG 810115.
\section{Preliminaries}

The following lemma is well known, see for example \cite[Proposition 5.2.]{meszaros2020distribution}.
\begin{lemma}\label{lemmacokmoment}
Let $\mathbf{M}$ be an $n\times n$ random matrix over $\mathbb{Z}$, and let $G$ be a finite abelian group. Then
\[\mathbb{E}|\Hom(\cok(\mathbf{M}),G)|=\sum_{\mathbf{g}\in G^n} \mathbb{P}(\mathbf{M}\mathbf{g}=0).\]

\end{lemma}

For a vector $\mathbf{g}\in G^n$, let $\langle\mathbf{g} \rangle$ be the subgroup of $G$ generated by the components of $\mathbf{g}$.

\begin{lemma}\label{lemmabalanced}
Let $G$ be a finite abelian group, let $\mathcal{P}$ be the set of prime divisors of $|G|$. There is a $c>0$ such that for all large enough $n$, if $\mathbf{M}$ is a $(\mathcal{P},\varepsilon)$-balanced $n\times n$ random matrix and $G_0$ is a subgroup of $G$, then
\[\exp(-\exp(-cn))\le \sum_{\substack{\mathbf{g}\in G_0^n\\\langle\mathbf{g}\rangle=G_0}} \min_{\mathbf{f}\in G_0^n} \mathbb{P}(\mathbf{M}\mathbf{g}=\mathbf{f})\le \sum_{\substack{\mathbf{g}\in G_0^n\\\langle\mathbf{g}\rangle=G_0}} \max_{\mathbf{f}\in G_0^n} \mathbb{P}(\mathbf{M}\mathbf{g}=\mathbf{f})\le \exp(\exp(-cn)).\]

\end{lemma}
\begin{proof}
Wood~\cite[Theorem 2.9]{wood2019random} proved that there are $c_0,K$ depending only on $G$ and $\varepsilon$ such that for any $(\mathcal{P},\varepsilon)$-balanced $n\times n$ random  matrix $\mathbf{M}$, we have
 \[1-K\exp(-c_0 n)\le \sum_{\substack{\mathbf{g}\in G^n\\\langle\mathbf{g}\rangle=G}} \mathbb{P}(\mathbf{M}\mathbf{g}=0)\le 1+K\exp(-c_0 n).\]
 One can see that the same proof actually gives the stronger statement that
 \begin{equation}\label{generalWood}1-K\exp(-c_0 n)\le \sum_{\substack{\mathbf{g}\in G^n\\\langle\mathbf{g}\rangle=G}} \min_{\mathbf{f}\in G^n} \mathbb{P}(\mathbf{M}\mathbf{g}=\mathbf{f})\le \sum_{\substack{\mathbf{g}\in G^n\\\langle\mathbf{g}\rangle=G}} \max_{\mathbf{f}\in G^n} \mathbb{P}(\mathbf{M}\mathbf{g}=\mathbf{f})\le 1+K\exp(-c_0 n).\end{equation}
 Indeed, one only needs to notice that the upper bound in \cite[Lemma 2.7]{wood2019random} on $\mathbb{P}(FX=0)$ also an upper bound on $\mathbb{P}(FX=g)$ for all $g\in G$. See \cite[Lemma 3.7]{nguyen2024rank}, where a detailed proof of this statement is written out. For the convenience of the reader, we repeat the argument of Wood~\cite{wood2019random} with the necessary changes in Section \ref{appendix2} of the Appendix.
 
 For large enough $n$, we have
 \[\exp(-2K\exp(-c_0n))\le 1-K\exp(-c_0 n)\le 1+K\exp(-c_0 n)\le \exp(2K\exp(-c_0n)).\]
 Thus, if we choose $0<c<c_0$, then for any large enough $n$, we have
 \[\exp(-\exp(-cn))\le \sum_{\substack{\mathbf{g}\in G^n\\\langle\mathbf{g}\rangle=G}} \min_{\mathbf{f}\in G^n} \mathbb{P}(\mathbf{M}\mathbf{g}=\mathbf{f})\le \sum_{\substack{\mathbf{g}\in G^n\\\langle\mathbf{g}\rangle=G}} \max_{\mathbf{f}\in G^n} \mathbb{P}(\mathbf{M}\mathbf{g}=\mathbf{f})\le \exp(\exp(-cn)).\]
The same statement is true for any subgroup $G_0$ of $G$ in place of $G$ with a possibly different constant~$c$. Taking the minimum of these constants, the statement follows. 
\end{proof}
The proof of the next lemma is straightforward.
\begin{lemma}\label{lemma6}
Let $G$ be a finite abelian group, let $\mathcal{P}$ be the set of prime divisors of $|G|$. Let $\mathbf{M}$ be a $(\mathcal{P},\varepsilon)$-balanced $m\times n$ random matrix and $G_0$ be a subgroup of $G$. Furthermore, let $\mathbf{g}$ be a deterministic vector in $G^n$ such that $\langle \mathbf{g}\rangle\not\subseteq G_0$. Then for any $\mathbf{f}\in G^m$, we have
\[\mathbb{P}(\mathbf{f}+\mathbf{M}\mathbf{g}\in G_0^m)\le (1-\varepsilon)^m.\]
\end{lemma}

Theorem~\ref{thmmain} follows from Theorem~\ref{homlemma} by using the following theorem of Nguyen and Van Peski~\cite{nguyen2024rank}.

\begin{theorem}{\normalfont(Nguyen and Van Peski~\cite[Theorem 1.2 and Proposition 5.3.]{nguyen2024rank})} \label{rescaledmomentmethod}Fix a prime $p$ and a positive integer $d$. Let $(\Gamma_h)_{h\ge 1}$ be a sequence of random
finitely-generated abelian $p$-groups and $(k_h)_{h\ge 1}$ a sequence of real numbers such that the following holds:
\begin{enumerate}[label=(\roman*)]
\item The fractional part $\{-\log_p(k_h)\}$ converges to $\zeta$.
\item For all $\lambda\in \Sig_d^{\ge 0}$, we have
\[\lim_{h\to\infty} \mathbb{E}\
\frac{|\Hom(\Gamma_h,G_{\lambda'})|}{k_h^{|\lambda|}}=\frac{c(G_{\lambda'},|\lambda|)}{|\lambda|!}.\]

\end{enumerate}

Then
\[\left(\rang(p^{i-1}\Gamma_h)-\lfloor\log_p(k_{h})+\zeta\rceil\right)_{i=1}^d\]
converges to $\mathcal{L}_{d,p^{-1},\chi}$ in distribution as $h\to\infty$, where $\chi=\frac{p^{-\zeta}}{p-1}$.
 
\end{theorem}
\section{The proof of Theorem~\ref{homlemma}}

\subsection{The outline of the proof}

For simplicity of notation, we drop the index $h$ whenever it does not cause any confusion.

Let $\mathbf{g}\in G^n$. We can write $\mathbf{g}$ as the concatenation of the vectors $\mathbf{g}_1,\mathbf{g}_2,\dots,\mathbf{g}_k$, where $\mathbf{g}_i\in G^{n_i}$. 

We define
\begin{align*}
    w(\mathbf{g})&=\left|\left\{i\in\{2,3,\dots,k\}\,:\,\langle \mathbf{g}_{i-1}\rangle\not\subseteq \langle\mathbf{g}_i\rangle\right \}\right|,\\
    t(\mathbf{g})&=\left|\left\{i\in\{1,2,\dots,k\}\,:\,\langle \mathbf{g}_{i-1}\rangle\subsetneq \langle\mathbf{g}_i\rangle\right \}\right|,
\end{align*}
with the convention that $\langle \mathbf{g}_0\rangle=\{0\}$.

By Lemma~\ref{lemmacokmoment}, we need to find the asymptotics of the sum
\begin{equation}\label{keysum}\sum_{\mathbf{g}\in G^n} \mathbb{P}(\mathbf{C}\mathbf{g}=0).\end{equation}

We split the sum according to the values of $w(\mathbf{g})$ and $t(\mathbf{g})$, and prove the following statements:
\begin{itemize}
    \item The vectors $\mathbf{g}\in G^n$ such that $w(\mathbf{g})>0$ have a negligible contribution to the sum in~\eqref{keysum}. See equation \eqref{wge1}.
    \item If $w(\mathbf{g})=0$, then $t(\mathbf{g})$ can take the values $0,1,\dots,\ell(G)$. Given an $i\in \{0,1,\dots,\ell(G)\}$, the total contribution of the vectors $\mathbf{g}\in G^n$ such that $w(\mathbf{g})=0$ and $t(\mathbf{g})=i$ to the sum is
    \[(1+o(1))c(G,i)\frac{k^i}{i!},\]
    see equation~\eqref{w0ti}.
    Here, the choice of $i=\ell(G)$ gives us the dominant contribution.
\end{itemize}
Once we prove these statements, Theorem~\ref{homlemma} follows easily.

\subsection{Details of the proof}

Let $\mathbf{B}'$ be a deterministic matrix in the range of $\mathbf{B}$. Then by \ref{A1}, \ref{A2}, \ref{A3} and \ref{A4}, we have
\begin{equation}\label{Pprod}
 \mathbb{P}\left(\mathbf{C}\mathbf{g}=0\,|\,\mathbf{B}=\mathbf{B}'\right)
 =\mathbb{P}\left(\mathbf{A}_{1,1}\mathbf{g}_1=0\right)\prod_{i=2}^k \mathbb{P}\left(\mathbf{A}_{i,i}\mathbf{g}_i=-\mathbf{A}_{i,i-1}\mathbf{g}_{i-1}-\sum_{j=1}^{i-2}\mathbf{B}'_{i,j}\mathbf{g}_j\right).
\end{equation}

Note that 
\begin{align}\label{Pcond}\mathbb{P}&\left(\mathbf{A}_{i,i}\mathbf{g}_i=-\mathbf{A}_{i,i-1}\mathbf{g}_{i-1}-\sum_{j=1}^{i-2}\mathbf{B}'_{i,j}\mathbf{g}_j\right)\\&=\mathbb{P}\left(\mathbf{A}_{i,i-1}\mathbf{g}_{i-1}+\sum_{j=1}^{i-2}\mathbf{B}'_{i,j}\mathbf{g}_j\in \langle \mathbf{g}_i\rangle^{n_i}\right)\nonumber\\&\qquad\cdot\mathbb{P}\left(\mathbf{A}_{i,i}\mathbf{g}_i=-\mathbf{A}_{i,i-1}\mathbf{g}_{i-1}-\sum_{j=1}^{i-2}\mathbf{B}'_{i,j}\mathbf{g}_j\,\Big|\, \mathbf{A}_{i,i-1}\mathbf{g}_{i-1}+\sum_{j=1}^{i-2}\mathbf{B}'_{i,j}\mathbf{g}_j\in \langle \mathbf{g}_i\rangle^{n_i}\right)\nonumber.\end{align}

If $\langle \mathbf{g}_{i-1}\rangle\not\subseteq \langle\mathbf{g}_i\rangle$, then by Lemma~\ref{lemma6}, we have the estimate
\[\mathbb{P}\left(\mathbf{A}_{i,i-1}\mathbf{g}_{i-1}+\sum_{j=1}^{i-2}\mathbf{B}'_{i,j}\mathbf{g}_j\in \langle \mathbf{g}_i\rangle^{n_i}\right)\le (1-\varepsilon)^{n_i}\le (1-\varepsilon)^{n_\bullet}.\]
If $\langle \mathbf{g}_{i-1}\rangle\subseteq \langle\mathbf{g}_i\rangle$, then trivially,
\[\mathbb{P}\left(\mathbf{A}_{i,i-1}\mathbf{g}_{i-1}+\sum_{j=1}^{i-2}\mathbf{B}'_{i,j}\mathbf{g}_j\in \langle \mathbf{g}_i\rangle^{n_i}\right)\le 1.\]
By \ref{A3}, the matrix $\mathbf{A}_{i,i}$ is independent from $\mathbf{A}_{i,i-1}$. Therefore,
\begin{multline*}\mathbb{P}\left(\mathbf{A}_{i,i}\mathbf{g}_i=-\mathbf{A}_{i,i-1}\mathbf{g}_{i-1}-\sum_{j=1}^{i-2}\mathbf{B}'_{i,j}\mathbf{g}_j\,\Big|\, \mathbf{A}_{i,i-1}\mathbf{g}_{i-1}+\sum_{j=1}^{i-2}\mathbf{B}'_{i,j}\mathbf{g}_j\in \langle \mathbf{g}_i\rangle^{n_i}\right)\\\le \max_{\mathbf{f}\in \langle \mathbf{g}_i\rangle^{n_i}}\mathbb{P}(\mathbf{A}_{i,i}\mathbf{g}_i=\mathbf{f}).
\end{multline*}

Thus,
\[ \mathbb{P}(\mathbf{C}\mathbf{g}=0\,|\,\mathbf{B}=\mathbf{B}')\le (1-\varepsilon)^{w(\mathbf{g})n_\bullet} \prod_{i=1}^n \max_{\mathbf{f}\in \langle \mathbf{g}_i\rangle^{n_i}}\mathbb{P}(\mathbf{A}_{i,i}\mathbf{g}_i=\mathbf{f}). \]

Therefore, by the law of total probability,
\begin{equation} \label{upperb}\mathbb{P}(\mathbf{C}\mathbf{g}=0)\le (1-\varepsilon)^{w(\mathbf{g})n_\bullet} \prod_{i=1}^n \max_{\mathbf{f}\in \langle \mathbf{g}_i\rangle^{n_i}}\mathbb{P}(\mathbf{A}_{i,i}\mathbf{g}_i=\mathbf{f}). \end{equation}

Let $\Sg(G)$ be the set of subgroups of $G$. For $\mathbf{g}\in G^n$, we introduce the notation \[\tyg(\mathbf{g})=(\langle \mathbf{g}_1\rangle,\langle \mathbf{g}_2\rangle,\dots,\langle \mathbf{g}_k\rangle)\in \Sg(G)^k.\]

Given a sequence $\mathbf{H}=(H_1,H_2,\dots,H_k)\in \Sg(G)^k$, let
\[\tyg^{-1}(\mathbf{H})=\{\mathbf{g}\in G^n\,:\,\tyg(\mathbf{g})=\mathbf{H}\}.\]

Let
\[W(\mathbf{H})=\left\{i\in\{2,3,\dots,k\}\,:\,{H}_{i-1}\not\subseteq {H}_i\right\}\text{ and }w(\mathbf{H})=|W(\mathbf{H})|.\]

Note that the definition $w(\mathbf{H})$  is consistent with our earlier definition of $w(\mathbf{g})$ in the sense that $w(\mathbf{g})=w(\tyg(\mathbf{g}))$.

 Assuming that $h$ is large enough
\begin{align}
\sum_{\mathbf{g}\in \tyg^{-1}(\mathbf{H})}& \mathbb{P}(C\mathbf{g}=0)\nonumber\\&\le (1-\varepsilon)^{w(\mathbf{H}) n_\bullet} \prod_{i=1}^k \sum_{\substack{\mathbf{g}_i\in H_i^{n_i}\nonumber\\\langle \mathbf{g}_i\rangle=H_i}} \max_{\mathbf{f}\in H_i^{n_i}}\mathbb{P}(\mathbf{A}_{i,i}\mathbf{g}_i=\mathbf{f})&(\text{by \eqref{upperb}})\nonumber\\&\le (1-\varepsilon)^{w(\mathbf{H}) n_\bullet} \prod_{i=1}^k \exp(\exp(-cn_i))&(\text{by \ref{A6} and Lemma~\ref{lemmabalanced}})\nonumber\\&\le (1-\varepsilon)^{w(\mathbf{H}) n_\bullet} \exp(k\exp(-cn_\bullet))\nonumber\\&\le 2(1-\varepsilon)^{w(\mathbf{H}) n_\bullet}\label{typupper}.&\text{(by \ref{A7})} 
\end{align}

Let $\mathbf{H}=(H_1,H_2,\dots,H_k)\in \Sg(G)^k$. For notational convenience let $H_0=\{0\}.$ We define
\[T(\mathbf{H})=\left\{i\in\{1,2,\dots,k\}\,:\,{H}_{i-1}\subsetneq {H}_i\right\}\text{ and }t(\mathbf{H})=|T(\mathbf{H})|.\]

Note that
\[\{1,2,\dots,k\}\setminus \left(W(\mathbf{H})\cup T(\mathbf{H})\right)=\left\{i\in\{1,2,\dots,k\}\,:\,{H}_{i-1}= {H}_i\right\}.\]

Thus, one can recover $\mathbf{H}$ from $W(\mathbf{H})\cup T(\mathbf{H})$ and the restriction of $\mathbf{H}$ to $W(\mathbf{H})\cup T(\mathbf{H})$. Therefore, it follows that for any $u$, we have
\begin{equation}\label{Hupper}\left|\left\{\,\mathbf{H}\in \Sg(G)^k\,:\,w(\mathbf{H})+t(\mathbf{H})=u\,\right\}\right|\le {{k}\choose{u}}|\Sg(G)|^u\le \frac{(|\Sg(G)|k)^u}{u!}.\end{equation}

\begin{claim}\label{claim7}
For any $\mathbf{H}\in \Sg(G)^k$, we have
\[t(\mathbf{H})\le \ell(G) (1+w(\mathbf{H})).\]
Consequently,
\[t(\mathbf{H})+w(\mathbf{H})< (\ell(G)+1) (1+w(\mathbf{H})).\]

\end{claim}
\begin{proof}
 Let $a<c$ be two consecutive elements of $W(\mathbf{H})$, and let us list the elements of $T(\mathbf{H})\cap [a,c]$ as $b_1<b_2<\cdots< b_r$. Then $\{0\}\subsetneq H_{b_1}\subsetneq H_{b_2}\subsetneq \cdots \subsetneq H_{b_r}$. In particular, one must have $r\le \ell(G)$. Thus, between any two consecutive elements of $W(\mathbf{H})$, there can be at most $\ell(G)$ elements of $T(\mathbf{H})$. A similar argument gives that there are at most $\ell(G)$ elements of $T(\mathbf{H})$ which are smaller than the smallest element of $W(\mathbf{H})$, and there are at most $\ell(G)$ elements of $T(\mathbf{H})$ which are larger than the largest element of $W(\mathbf{H})$. Thus, the statement follows.
\end{proof}

Combining \eqref{Hupper} and Claim~\ref{claim7}, we see that for any $w\ge 1$, we have
\begin{align*}
\left|\left\{\,\mathbf{H}\in \Sg(G)^k\,:\,w(\mathbf{H})=w\,\right\}\right|&\le \sum_{u=w}^{(\ell(G)+1)(w+1)} \frac{(|\Sg(G)|k)^u}{u!}\\
&\le \frac{(|\Sg(G)|k)^{(\ell(G)+1)(w+1)}}{w!} \sum_{i=0}^{\infty} \frac{1}{i!}\\
&\le 3 \frac{(|\Sg(G)|k)^{2(\ell(G)+1)w}}{w!}.
\end{align*}

Combining this with \eqref{typupper}, we see that for all $w\ge 1$, we have
\[
 \sum_{\substack{\mathbf{g}\in G^n\\w(\mathbf{g})=w}} \mathbb{P}(\mathbf{C}\mathbf{g}=0)\le 6 \frac{(|\Sg(G)|k)^{2(\ell(G)+1)w}}{w!} (1-\varepsilon)^{w n_\bullet},
\]
provided that $h$ is large enough. Therefore,
\begin{align*}\sum_{\substack{\mathbf{g}\in G^n\\w(\mathbf{g})\ge 1}} \mathbb{P}(\mathbf{C}\mathbf{g}=0)&\le 6 \sum_{w=1}^\infty \frac{\left((|\Sg(G)|k)^{2(\ell(G)+1)}(1-\varepsilon)^{n_\bullet}\right)^w}{w!} \\&=6\left(\exp\left((|\Sg(G)|k)^{2(\ell(G)+1)}(1-\varepsilon)^{n_\bullet}\right)-1\right).
\end{align*}

By assumptions \ref{A7} and \ref{A6}, we see that
\[\lim_{h\to\infty}(|\Sg(G)|k(h))^{2(\ell(G)+1)}(1-\varepsilon)^{n_\bullet(h)}=0.\]
Thus, 
\begin{equation}\label{wge1}
\lim_{h\to\infty} \sum_{\substack{\mathbf{g}\in G^{n(h)}\\w(\mathbf{g})\ge 1}} \mathbb{P}(\mathbf{C}^{(h)}\mathbf{g}=0)=0. 
\end{equation}

It is straightforward to see that
\begin{equation}\label{Hsize}\left|\left\{\,\mathbf{H}\in \Sg(G)^k\,:\,w(\mathbf{H})=0\text{ and }t(\mathbf{H})=i\,\right\}\right|=c(G,i){{k}\choose{i}}.\end{equation}

Note that if for a $\mathbf{g}\in G^n$, we have $w(\mathbf{g})=0$, then for all $i\ge 2$, we have
\[\mathbb{P}\left(\mathbf{A}_{i,i-1}\mathbf{g}_{i-1}+\sum_{j=1}^{i-2}\mathbf{B}'_{i,j}\mathbf{g}_j\in \langle \mathbf{g}_i\rangle^{n_i}\right)= 1.\]

Thus, as before it follows from \eqref{Pprod} and \eqref{Pcond} that
\[\mathbb{P}(\mathbf{C}\mathbf{g}=0)\ge \prod_{i=1}^n \min_{\mathbf{f}\in \langle \mathbf{g}_i\rangle^{n_i}}\mathbb{P}(\mathbf{A}_{i,i}\mathbf{g}_i=\mathbf{f}).\]

As before, this implies that for any $\mathbf{H}\in \Sg(G)^k$ such that $w(\mathbf{H})=0$, we have
\[\sum_{\mathbf{g}\in \tyg^{-1}(\mathbf{H})}\mathbb{P}(\mathbf{C}\mathbf{g}=0)\ge \exp(-k\exp(-cn_\bullet))=1+o_h(1),\]
where the last equality follows from \ref{A7}.

We have already seen that
\[\sum_{\mathbf{g}\in \tyg^{-1}(\mathbf{H})}\mathbb{P}(\mathbf{C}\mathbf{g}=0)\le \exp(k\exp(-cn_\bullet))=1+o_h(1),\]
so
\[\sum_{\mathbf{g}\in \tyg^{-1}(\mathbf{H})}\mathbb{P}(\mathbf{C}\mathbf{g}=0)=1+o_h(1).\]

Combining this with \eqref{Hsize} and \ref{A5}, we obtain
\begin{equation}\label{w0ti}
    \sum_{\substack{\mathbf{g}\in G^n\\w(\mathbf{g})=0\\t(\mathbf{g})=i}}\mathbb{P}(\mathbf{C}\mathbf{g}=0)=(1+o_h(1))c(G,i){{k}\choose{i}}=(1+o_h(1))c(G,i)\frac{k^i}{i!}.\end{equation}

Thus,
\begin{equation}\label{limitmoment9}\lim_{h\to\infty}(k(h))^{-\ell(G)}\sum_{\substack{\mathbf{g}\in G^{n(h)}\\w(\mathbf{g})=0\\t(\mathbf{g})=i}}\mathbb{P}(\mathbf{C}^{(h)}\mathbf{g}=0)=\begin{cases}\frac{c(G,\ell(G))}{\ell(G)!}&\text{if }i=\ell(G),\\0&\text{otherwise}.\end{cases}\end{equation}

Using Lemma~\ref{lemmacokmoment}, we see that
 \begin{align*}\frac{\mathbb{E}|\Hom(\cok(\mathbf{C}),G)|}{k^{\ell(G)}}&=k^{-\ell(G)}\sum_{\mathbf{g}\in G^n} \mathbb{P}(\mathbf{C}\mathbf{g}=0)\\&=k^{-\ell(G)}\left(\sum_{\substack{\mathbf{g}\in G^n\\w(\mathbf{g})\ge 1}} \mathbb{P}(\mathbf{C}\mathbf{g}=0)+\sum_{i=0}^{\ell(G)}\sum_{\substack{\mathbf{g}\in G^{n}\\w(\mathbf{g})=0\\t(\mathbf{g})=i}}\mathbb{P}(\mathbf{C}\mathbf{g}=0)\right).
 \end{align*}

Combining this with \eqref{wge1} and \eqref{limitmoment9}, Theorem~\ref{homlemma} follows.

\section{The proof of Theorem~\ref{improvedtheorem}}

The proof of Theorem~\ref{improvedtheorem} is very similar to the proof of Theorem~\ref{homlemma}. It is actually even easier, so we only sketch the details. Let $G$ be a finite abelian $p$-group. We drop the index of $k_n$. Let $\mathbf{A}_1,\mathbf{A}_2,\dots,\mathbf{A}_k$ be independent $(\{p\},\varepsilon)$-balanced $n\times n$ random matrices. We have
\begin{align*}\sum_{\mathbf{g}_k\in G^n}\mathbb{P}(\mathbf{A}_1\mathbf{A}_2\cdots\mathbf{A}_k\mathbf{g}_k=0)&=\sum_{\substack{\mathbf{g}_1,\mathbf{g}_2,\dots,\mathbf{g}_k\in G^n\\ \langle\mathbf{g}_1\rangle\subseteq \langle\mathbf{g}_2\rangle\subseteq\cdots\subseteq \langle\mathbf{g}_k\rangle }} \mathbb{P}(\mathbf{A}_i\mathbf{g}_i=\mathbf{g}_{i-1}\text{ for all }i=1,2,\dots,k)\\&=\sum_{\substack{\mathbf{g}_1,\mathbf{g}_2,\dots,\mathbf{g}_k\in G^n\\ \langle\mathbf{g}_1\rangle\subseteq \langle\mathbf{g}_2\rangle\subseteq\cdots\subseteq \langle\mathbf{g}_k\rangle }} \prod_{i=1}^k\mathbb{P}(\mathbf{A}_i\mathbf{g}_i=\mathbf{g}_{i-1}),
\end{align*}
with the notation that $\mathbf{g}_0=0$.

Let $H_1\subseteq H_2\subseteq \cdots\subseteq H_k\subseteq G$ be a chain of subgroups of $G$. Then, assuming that $n$ is large enough,
\begin{align*}
\sum_{\substack{\mathbf{g}_1\in H_1^n\\\langle\mathbf{g}_1\rangle=H_1}}\sum_{\substack{\mathbf{g}_2\in H_2^n\\\langle\mathbf{g}_2\rangle=H_2}}\cdots \sum_{\substack{\mathbf{g}_k\in H_k^n\\\langle\mathbf{g}_k\rangle=H_k}} &\prod_{i=1}^k\mathbb{P}(\mathbf{A}_i\mathbf{g}_i=\mathbf{g}_{i-1})\\&\le \sum_{\substack{\mathbf{g}_1\in H_1^n\\\langle\mathbf{g}_1\rangle=H_1}}\sum_{\substack{\mathbf{g}_2\in H_2^n\\\langle\mathbf{g}_2\rangle=H_2}}\cdots \sum_{\substack{\mathbf{g}_k\in H_k^n\\\langle\mathbf{g}_k\rangle=H_k}} \prod_{i=1}^k\max_{\mathbf{f}\in H_i^n}\mathbb{P}(\mathbf{A}_i\mathbf{g}_i=\mathbf{f})\\&=\prod_{i=1}^k \sum_{\substack{\mathbf{g}_i\in H_i^n\\\langle\mathbf{g}_i\rangle=H_i}}\max_{\mathbf{f}\in H_i^n}\mathbb{P}(\mathbf{A}_i\mathbf{g}_i=\mathbf{f})\\&\le \exp(k\exp(-cn)),
\end{align*}
where the last inequality follows from Lemma~\ref{lemmabalanced}. Similarly,
\[\sum_{\substack{\mathbf{g}_1\in H_1^n\\\langle\mathbf{g}_1\rangle=H_1}}\sum_{\substack{\mathbf{g}_2\in H_2^n\\\langle\mathbf{g}_2\rangle=H_2}}\cdots \sum_{\substack{\mathbf{g}_k\in H_k^n\\\langle\mathbf{g}_k\rangle=H_k}} \prod_{i=1}^k\mathbb{P}(\mathbf{A}_i\mathbf{g}_i=\mathbf{g}_{i-1})\ge \exp(-k\exp(-cn)).
\]
Using the assumption that $\log k=o(n)$, it follows that
\[\sum_{\substack{\mathbf{g}_1\in H_1^n\\\langle\mathbf{g}_1\rangle=H_1}}\sum_{\substack{\mathbf{g}_2\in H_2^n\\\langle\mathbf{g}_2\rangle=H_2}}\cdots \sum_{\substack{\mathbf{g}_k\in H_k^n\\\langle\mathbf{g}_k\rangle=H_k}} \prod_{i=1}^k\mathbb{P}(\mathbf{A}_i\mathbf{g}_i=\mathbf{g}_{i-1})=1+o(1).\]

Then following along the lines of the proof of Theorem~\ref{homlemma}, it can be proved that \[\lim_{n\to\infty} \frac{\mathbb{E}|\Hom(\Gamma_n,G)|}{k^{\ell(G)}}=\frac{c(G,\ell(G))}{\ell(G)!}.\]

Thus, Theorem~\ref{improvedtheorem} follows by using  Theorem~\ref{rescaledmomentmethod}.

\appendix

\section{Appendix}
\subsection{The proof of \eqref{cokidentity}}\label{appendix}

We prove by induction on $k$. Consider the following block matrices with $k\times k$ blocks, where each block is of size $n\times n$:
\[\mathbf{S}=\begin{pmatrix}
\mathbf{I}&\mathbf{0}&\ddots\\
\mathbf{0}&\mathbf{I}&\mathbf{0}&\ddots\\
\ddots&\ddots&\ddots&\ddots&\ddots&\\
&\ddots&\mathbf{0}&\mathbf{I}&\mathbf{0}&\ddots&\\
&&\ddots&\mathbf{0}&\mathbf{I}&-\mathbf{A}_{k-1}\\
&&&\ddots&\mathbf{0}&\mathbf{I}\\
\end{pmatrix},\qquad \mathbf{T}=\begin{pmatrix}
\mathbf{I}&\mathbf{0}&\ddots\\
\mathbf{0}&\mathbf{I}&\mathbf{0}&\ddots\\
\ddots&\ddots&\ddots&\ddots&\ddots&\\
&\ddots&\mathbf{0}&\mathbf{I}&\mathbf{0}&\ddots&\\
&&\ddots&\mathbf{0}&\mathbf{A}_k&\mathbf{I}\\
&&&\ddots&-\mathbf{I}&\mathbf{0}\\
\end{pmatrix},\]
\[\mathbf{C}'=\begin{pmatrix}
\mathbf{A}_1&\mathbf{0}&\ddots&&&\\
\mathbf{I}&\mathbf{A}_2&\mathbf{0}&\ddots&\\
\mathbf{0}&\mathbf{I}&\mathbf{A}_3&\mathbf{0}&\ddots&\\
\ddots&\ddots&\ddots&\ddots&\ddots&\ddots\\
 &\ddots&\ddots&\mathbf{I}&\mathbf{A}_{k-2}&\mathbf{0}&\ddots\\
& &\ddots&\ddots&\mathbf{I}&\mathbf{A}_{k-1}\mathbf{A}_k&\mathbf{0}\\
& &\ddots&\ddots&\mathbf{0}&\mathbf{0}&\mathbf{I}
\end{pmatrix}.\]

Note that $\mathbf{S},\mathbf{T}\in \text{GL}(nk,\mathbb{Z})$ and $\mathbf{SCT}=\mathbf{C}'$. Thus, it follows easily from the definition of the cokernel, that $\cok(\mathbf{C})\cong \cok(\mathbf{C}')$, and also
\[\cok(\mathbf{C}')\cong\cok\begin{pmatrix}
\mathbf{A}_1&\mathbf{0}&\ddots&&&\\
\mathbf{I}&\mathbf{A}_2&\mathbf{0}&\ddots&\\
\mathbf{0}&\mathbf{I}&\mathbf{A}_3&\mathbf{0}&\ddots&\\
\ddots&\ddots&\ddots&\ddots&\ddots&\ddots\\
 &\ddots&\ddots&\mathbf{I}&\mathbf{A}_{k-2}&\mathbf{0}&\\
& &\ddots&\ddots&\mathbf{I}&\mathbf{A}_{k-1}\mathbf{A}_k
\end{pmatrix}.\]
Thus, the statement follows by induction.

\subsection{The proof of \eqref{generalWood}}\label{appendix2}

The proof that we present here is almost the same as the proof of \cite[Theorem 2.9]{wood2019random}. Note that we changed some of the notations of \cite{wood2019random} to be consistent with the other parts of the present paper.

We say that $\mathbf{g}\in G^n$ is a code of distance $w$, if for all $X\subset \{1,2,\dots,n\}$ such that $|X|<w$, we have
\[\langle \mathbf{g}(i)\,:\,i\in \{1,2,\dots,n\}\setminus X\rangle = G.\]

\begin{lemma}\label{codes}
Let $\delta>0$. Then there are constants $c,K>0$ depending only on $G$, $\varepsilon$ and $\delta$ with the following property. Let $\mathbf{M}$ be a $(\mathcal{P},\varepsilon)$-balanced $n\times n$ random matrix and let $\mathbf{g}\in G^n$ be a code of distance $\delta n$, then 
\[\left|\mathbb{P}(\mathbf{Mg}=\mathbf{f})-|G|^{-n}\right|\le \frac{K\exp(-cn)}{|G|^n} \qquad\text{ for all}\quad\mathbf{f}\in G^n.\]
\end{lemma}
\begin{proof}
    See \cite[Lemma 2.4.]{wood2019random}.
\end{proof}

For $D=\prod_i p_i^{e_i}$, we define $\ell(D)=\sum_i e_i$. This is consistent with our definition of $\ell(G)$ in the Introduction in the sense that $\ell(G)=\ell(|G|)$. 

Given $\delta>0$ and $\mathbf{g}\in G^n$, the $\delta$-depth of $\mathbf{g}$ is the maximal positive $D$  such that there is an $X\subset\{1,2,\dots,n\}$ satisfying $|X|<\ell(D)\delta n$  and
\[\left[G\,:\,\langle \mathbf{g}(i)\,:\,i\in \{1,2,\dots,n\}\setminus X\rangle\right]=D.\]

If there is not any such $D$, the $\delta$-depth of $\mathbf{g}$ is defined to be $1$.

The proof of the next lemma is straightforward.

\begin{lemma}\label{countdepth}
There is a constant $K$ depending $G$ such that for all $D>1$, we have at most
\[K{{n}\choose{\lceil\ell(D)\delta n-1\rceil}} |G|^n D^{-n+\ell(D)\delta n}\]
vectors in $G^n$ with $\delta$-depth $D$.
\end{lemma}

\begin{lemma}\label{probdepthestimate}
Let $\delta>0$. Then there is a $K$ depending only on $G$, $\varepsilon$ and $\delta$ with the following property. Assume that $\mathbf{g}\in G^n$ has $\delta$-depth $D>1$ and $\langle\mathbf{g} \rangle=G$.  Let $\mathbf{M}$ be   $(\mathcal{P},\varepsilon)$-balanced $n\times n$ random matrix and $\mathbf{f}\in G^n$, then
\[\mathbb{P}(\mathbf{Mg}=\mathbf{f})\le K\exp(-\varepsilon n)\frac{D^n}{|G|^n}.\]
\end{lemma}
\begin{proof}
    See \cite[Lemma 3.8.]{nguyen2024rank}.
\end{proof}

The following lemma is equivalent to \eqref{generalWood}.
\begin{lemma}
There are constants $K,c>0$ depending only on $G$ and $\varepsilon$ with the following property.
Let $\mathbf{M}$ be an $(\mathcal{P},\varepsilon)$-balanced $n\times n$ random matrix. For $\mathbf{g}\in G^n$, let us define 
\[P(\mathbf{g})=\max_{\mathbf{f}\in G^n} \mathbb{P}(\mathbf{Mg}=\mathbf{f}).\]
Then
\[\left|\sum_{\substack{\mathbf{g}\in G^n\\\langle\mathbf{g}\rangle=G}}P(\mathbf{g})-1\right|\le K\exp(-cn).\]
The same is true if we replace the definition of $P(\mathbf{g})$ with 
\[P(\mathbf{g})=\min_{\mathbf{f}\in G^n} \mathbb{P}(\mathbf{Mg}=\mathbf{f}).\]
\end{lemma}
\begin{proof}
The proof follows the proof of \cite[
Theorem 2.9]{wood2019random} almost verbatim. In this proof $K$ will be a constant depending only on $\varepsilon,\delta$ and $G$, and we allow it to change from line to line.

Using Lemmas~\ref{countdepth} and \ref{probdepthestimate}, we have
\begin{align*}
\sum_{\substack{\mathbf{g}\in G^n,\, \langle \mathbf{g} \rangle=G,\\\textrm{ $\mathbf{g}$ is not a code of distance $\delta n$}
} }
P(\mathbf{g})&\leq
\sum_{\substack{D>1\\ D\mid\#G}} \quad \sum_{\substack{\mathbf{g}\in G^n,\, \langle \mathbf{g}\rangle=G,\\ \mathbf{g}\textrm{ has  $\delta$-depth $D$}
} }
P(\mathbf{g})\\
&\leq  \sum_{\substack{D>1\\ D\mid\#G}} K\binom{n}{\lceil \ell(D)\delta n \rceil -1}
|G|^{n}D^{-n+\ell(D)\delta n} \exp(-\epsilon n) D^n|G|^{-n}\\
%
&\leq  \sum_{\substack{D>1\\ D\mid\#G}} K\binom{n}{\lceil \ell(D)\delta n \rceil -1}
D^{\ell(D)\delta n} \exp(-\epsilon n) \\
&\leq K\binom{n}{\lceil \ell(|G|)\delta n \rceil -1}
|G|^{\ell(|G|)\delta n} \exp(-\epsilon n) \\
 &\leq  K
 e^{-dn},
\end{align*}
provided that $\delta$ is small enough.

Also, from Lemma~\ref{countdepth}, we can choose $\delta$ small enough so that we have
\begin{align*}
\sum_{\substack{\mathbf{g}\in G^n,\,\langle\mathbf{g}\rangle=G\\ \mathbf{g}\textrm{ is not  a code of distance $\delta n$}
} }
|G|^{-n}&\leq
\sum_{\substack{D>1\\ D\mid\#G}} \quad \sum_{\substack{\mathbf{g}\in G^n,\,\langle\mathbf{g}\rangle=G\\ \mathbf{g}\textrm{ has $\delta$-depth $D$}
} }
|G|^{-n}\\
&\leq \sum_{\substack{D>1\\ D\mid\#G}}  K\binom{n}{\lceil \ell(D)\delta n \rceil -1} |G|^n|D|^{-n+\ell(D)\delta n}|G|^{-n}
\\
&\leq K \binom{n}{\lceil \ell(|G|)\delta n \rceil -1}2^{-n+\ell(|G|)\delta n}  \\
 &\leq  K
 e^{-dn}.
\end{align*}

We also have
 \begin{align*}
\sum_{\substack{\mathbf{g}\in G^n\\\langle g\rangle\neq G
} }
|G|^{-n}&\leq
\sum_{G \neq H\in \Sg(G)}\quad \sum_{\substack{\mathbf{g}\in H^n\\\langle g\rangle=H} }
|G|^{-n}\\
&\leq
\sum_{G \neq H\in \Sg(G)} |H|^{n}
|G|^{-n}\\
&\leq K
 e^{-dn}.
\end{align*}

Then given a choice of $\delta$ that satisfies the requirements above, using Lemma~\ref{codes} we have a $c$ such that
\begin{align*}
\sum_{\substack{\mathbf{g}\in G^n\\ \mathbf{g}\textrm{ is a code of distance $\delta n$}
} }
\left| P(\mathbf{g}) -|G|^{-n}\right|
&\leq
 Ke^{-cn} .
\end{align*}
If necessary, we take $c$ smaller so $c \leq d$.
Therefore,
\begin{align*}
&\left|\sum_{\substack{\mathbf{g}\in G^n\\\langle\mathbf{g}\rangle=G}}P(\mathbf{g})-1\right| =\left|\sum_{\substack{\mathbf{g}\in G^n\\\langle\mathbf{g}\rangle=G}}P(\mathbf{g})-\sum_{\mathbf{g}\in G^n}|G|^{-n}\right|  \\
&\leq
\sum_{\substack{\mathbf{g}\in G^n\\ \mathbf{g}\textrm{ is a code of dist. $\delta n$}
} }
\left| P(\mathbf{g}) -|G|^{-n}\right|
+
 \sum_{\substack{\mathbf{g}\in G^n,\, \langle \mathbf{g} \rangle=G,\\\textrm{ $\mathbf{g}$ is not a code of dist. $\delta n$}
} }
P(\mathbf{g})  + \sum_{\substack{\mathbf{g}\in G^n\\ \mathbf{g}\textrm{ is not a code of dist. $\delta n$}
} } |G|^{-n} \\
\\
&\leq Ke^{-cn}. \qedhere
\end{align*}

\end{proof}

\bibliography{references}
\bibliographystyle{plain}

\bigskip

\bigskip

\noindent Andr\'as M\'esz\'aros, \\
HUN-REN Alfr\'ed R\'enyi Institute of Mathematics, \\Budapest, Hungary,\\ {\tt meszaros@renyi.hu}
\end{document}